\documentclass[10pt, letterpaper,reqno]{amsart}
\usepackage{amssymb}
\usepackage{amsmath}
\usepackage{amsfonts, color}
\usepackage{graphicx}
\usepackage{mathrsfs,amssymb, slashed, cite}

\usepackage{hyperref}
\usepackage{cleveref}

\let\Re=\undefined\DeclareMathOperator*{\Re}{Re}
\let\Im=\undefined\DeclareMathOperator*{\Im}{Im}

\newcommand{\R}{\mathbb{R}}
\newcommand{\C}{\mathbb{C}}

\newtheorem{corollary}{Corollary}[section]

\newtheorem{lemma}{Lemma}[section]

\newtheorem{proposition}{Proposition}[section]

\newtheorem{theorem}{Theorem}[section]
\numberwithin{equation}{section}

\newcommand{\eps}{\varepsilon}

\newcommand{\qtq}[1]{\quad\text{#1}\quad}

\allowdisplaybreaks

\begin{document}

\title[Approximate traveling waves]{Approximate traveling waves for the \\ focusing NLS with an external potential}

\author[L. Baker]{Luke Baker}
\email{lukebake@uoregon.edu}
\address{Department of Mathematics, University of Oregon, Eugene, OR, USA}

\author[X. Li]{Xuemei Li}
\email{xuemei\_li@mail.bnu.edu.cn}
\address{Laboratory of Mathematics and Complex Systems, Ministry of Education, School of Mathematical Sciences, Beijing Normal University, Beijing 100875, People’s Republic of China}

\author[J. Murphy]{Jason Murphy}
\email{jamu@uoregon.edu}
\address{Department of Mathematics, University of Oregon, Eugene, OR, USA}

\maketitle

\begin{abstract} We construct approximate traveling wave solutions for the $3d$ focusing cubic NLS perturbed by a rapidly decaying external potential.  As an application, we demonstrate that there exists a forward-global, $H^1$ bounded, nonscattering solution at any mass-energy higher than the ground state threshold for the unperturbed equation.
\end{abstract}

\section{Introduction}\label{S:intro}

We consider the focusing cubic nonlinear Schr\"odinger equation (NLS) in three dimensions in the presence of a real-valued external potential:
\begin{equation}\label{NLS}
i\partial_t u = -\Delta u + Vu - |u|^2 u,
\end{equation}
which we view as a perturbation of the standard focusing NLS:
\begin{equation}\label{NLS0}
i\partial_t u = -\Delta u - |u|^2 u.
\end{equation}

We restrict attention to nonnegative, repulsive, rapidly decaying potentials.  In particular, we consider $V\in C^1(\R^3;\R)$ such that
\begin{equation}\label{V}
V\geq 0,\quad x\cdot\nabla V\leq 0,\qtq{and} \forall c>0\  \sup_{x\in\R^3} e^{c|x|}\bigl\{|V(x)|+|\nabla V(x)|\bigr\} <\infty.
\end{equation}

The \emph{ground state} $Q$ is the unique nonnegative, radially decreasing solution to the elliptic equation 
\begin{equation}\label{Qeq}
-Q = -\Delta Q - Q^3,
\end{equation}
which produces the ground state solution $e^{it}Q(x)$ to \eqref{NLS0}.  More generally, given any $\alpha>0$ and $\xi\in\R^3$, we can produce the following traveling wave solution to \eqref{NLS0}:
\begin{equation}\label{W1}
W_0(t,x) =\alpha  e^{i[\alpha^2 t + x\cdot\xi - |\xi|^2 t]}Q(\alpha[x-2t\xi]). 
\end{equation}

Our goal in this work is to construct forward-global solutions to \eqref{NLS} that behave like traveling waves of the form \eqref{W1} as $t\to\infty$. We will discuss the motivation and context for this problem below. 

As we will explain, our techniques will require us to impose some restrictions on the traveling wave parameters $(\alpha,\xi)$. To state these restrictions precisely, we first introduce a parameter $\delta_0>0$ such that 
\[
|\Psi(x)| \lesssim e^{-\delta_0|x|}\qtq{for}\Psi\in\{Q,\nabla Q,Y_1^\pm, \nabla Y_1^\pm\},
\] 
where $Y_1^\pm$ are the $L^2$-normalized eigenfunctions for the operator arising from the linearization of \eqref{NLS0} around $Q$ (see Section~\ref{S:notation}).  We denote the eigenvalues corresponding to $Y_1^\pm$ by $\pm\lambda_1$.  Given $V$ satisfying \eqref{V} and $\delta\in(0,\delta_0)$, we then define
\begin{equation}\label{Pdelta}
\mathcal{P}_\delta:=\biggl\{(\alpha,\xi)\in(1,\infty)\times\R^3\backslash\{0\}: \alpha>\max\biggl\{ \frac{C\|\nabla V\|_{L^4}}{\delta|\xi|},\frac{\delta|\xi|}{\lambda_1}\biggr\}\biggr\}\footnote{Here $C\geq 1$ is a universal constant.}, 
\end{equation}
and we define the set of admissible traveling wave parameters by
\[
\mathcal{P}=\bigcup_{\delta\in(0,\delta_0)}\mathcal{P}_\delta.
\] 

Our main result is the following theorem.

\begin{theorem}\label{T} Suppose $V\in C^1(\R^3;\R)$ satisfies \eqref{V}.  For any $(\alpha,\xi)\in\mathcal{P}$, there exists $T_0>0$ and a solution $u:[T_0,\infty)\times\R^3\to\C$ to \eqref{NLS} such that 
\begin{equation}\label{H1C}
\lim_{t\to\infty} \| u(t) - W_0(t) \|_{H^1(\R^3)} = 0,\footnote{If $(\alpha,\xi)\in\mathcal{P}_\delta$, then the convergence is obtained with a rate of $e^{-\frac12\delta\alpha|\xi| t}$.}
\end{equation}
where
\[
W_0(t,x)=\alpha e^{i[\alpha^2 t + x\cdot\xi - |\xi|^2 t]}Q(\alpha[x-2t\xi]).
\]
\end{theorem}

\smallskip

\emph{Context and motivation.} 

\smallskip

Our main result is motivated by the study of the long-time dynamics of NLS with an external potential. Recent work in nonlinear dispersive PDE has aimed to classify the dynamics of solutions according to their orientation in the mass-energy plane (see e.g. \cite{DuyckaertsRoudenko, DHR, MMZ, Hong, DLR, KMVZ, KOPV, KMV}).  Here the conserved \emph{mass} of solutions to \eqref{NLS} is defined by 
\[
M(u) = \int_{\R^3} |u|^2\,dx,
\]
while the conserved \emph{energy} is given by 
\[
E(u) = E_0(u) + \tfrac12\int_{\R^3} V|u|^2\,dx,\qtq{where} E_0(u) := \int_{\R^3} \tfrac12|\nabla u|^2 - \tfrac14|u|^4\,dx.
\]
In particular, $E_0$ is the conserved energy for the standard NLS \eqref{NLS0}. For later use, let us also define the \emph{momentum}, which we note is conserved for \eqref{NLS0} but \emph{not} for \eqref{NLS}:
\[
P(u):=2\Im \int \bar u \nabla u\,dx.
\]

For potentials satisfying \eqref{V}, one can obtain the following theorem. 

\begin{theorem}\label{T:Hong} Let $V\in C^1(\R^3;\R)$ (not identically zero) satisfy \eqref{V}. Suppose that $u_0\in H^1(\R^3)$ satisfies
\begin{equation}\label{me1}
M(u_0)E(u_0) \leq M(Q)E_0(Q),
\end{equation}
and let $u$ be the corresponding solution to \eqref{NLS}. 

(i) If
\begin{equation}\label{me2}
\|u_0\|_{L^2}\|(-\Delta+V)^{\frac12}u_0\|_{L^2} < \|Q\|_{L^2}\|\nabla Q\|_{L^2},
\end{equation}
then $u$ scatters\footnote{We say $u$ scatters in $H^1$ if there exist $u_\pm\in H^1$ such that $\|u(t)-e^{it\Delta}u_\pm\|_{H^1}=0$. In the context of \eqref{NLS} with $V$ satisfying \eqref{V}, scattering is equivalent to the finiteness of the global $L_{t,x}^5$ space-time norm.} in both time directions. 

(ii) If 
\begin{equation}\label{me3}
\|u_0\|_{L^2}\|(-\Delta+V)^{\frac12}u_0\|_{L^2} > \|Q\|_{L^2}\|\nabla Q\|_{L^2},
\end{equation}
then $u$ blows up its $H^1$ norm in both time directions. If $xu_0\in L^2$, then we can guarantee blowup in finite time.

(iii) Equality in \eqref{me2} or \eqref{me3} under the assumption \eqref{me1} is impossible.  In particular, there is no ground state for \eqref{NLS}, in the sense that the sharp Galigardo--Nirenberg inequality
\[
\|f\|_{L^4(\R^3)} \leq C_0 \|f\|_{L^2(\R^3)}^{\frac14} \|(-\Delta+V)^{\frac12}f\|_{L^2(\R^3)}^{\frac34} 
\]
admits no optimizer in $H^1(\R^3)\backslash\{0\}$. 

(iv) There exists a sequence of global solutions $u_n$ to \eqref{NLS} satisfying
\begin{align*}
&M(u_n)E(u_n)\nearrow M(Q) E_0(Q),\\
& \|u_n(0)\|_{L^2}\|(-\Delta+V)^{\frac12}u_n(0)\|_{L^2} \nearrow \|Q\|_{L^2}\|\nabla Q\|_{L^2},
\end{align*} 
and
\[
\lim_{n\to\infty} \|u_n\|_{L_{t,x}^5(\R\times\R^3)}=\infty. 
\]
\end{theorem}

The proof of Theorem~\ref{T:Hong} may be found in the works \cite{Hong, MMZ}. Items (i) and (ii) show that NLS in the presence of a nonnegative, repulsive potential admits the same mass-energy threshold for a \emph{scattering/blow-up dichotomy} as the standard NLS (cf. \cite{DHR}).  In fact, in contrast to the standard NLS, the dichotomy persists at the threshold.  In particular, just as one finds no ground state at the mass-energy of $Q$ (item (iii)), one finds no analogues of the \emph{heteroclinic orbits} $Q^\pm$ arising in the analysis of threshold solutions to \eqref{NLS0} (see e.g. \cite{DuyckaertsRoudenko}).  Nonetheless, the existence of the solutions in part (iv) suggests that $M(Q)E_0(Q)$ should be the optimal threshold for obtaining a simple classification of dynamics into blowup versus scattering.  

Our main result, Theorem~\ref{T}, confirms that there can be no simple scattering/blowup dichotomy above the mass-energy of $Q$. In particular, we have the following corollary to Theorem~\ref{T},  in which we construct approximate traveling waves at any mass-energy above that of $Q$.

\begin{corollary}\label{C} Suppose $V\in C^1(\R^3;\R)$ satisfies \eqref{V}. For any $L>M(Q)E_0(Q)$, there exists a forward-global, $H^1$-bounded, non-scattering solution $u$ to \eqref{NLS} with $M(u) E(u) = L$.
\end{corollary}

\begin{proof} Suppose $u:[T_0,\infty)\times\R^3\to\C$ is the solution given by Theorem~\ref{T} for some $(\alpha,\xi)\in\mathcal{P}$, and set
\[
W_0(t,x) = \alpha e^{i[\alpha^2 t + x\cdot\xi - |\xi|^2 t]}Q(\alpha[x-2t\xi]).
\]
We note that $u$ is forward-global and bounded in $H^1$, but claim that $u$ does not scatter in $H^1$. Indeed, if $u$ scatters forward in time then one can readily obtain that $\|u(t)\|_{L^6}\to 0$ as $t\to\infty$, which is incompatible with the convergence in \eqref{H1C}.

Using conservation of mass and energy and \eqref{H1C}, we now observe that
\begin{align*}
M(u) E(u)  &= \lim_{t\to\infty} M(W_0(t)) E(W_0(t)) \\
& = \lim_{t\to\infty} M(W_0(t)) E_0(W_0(t)) \\
& = M(Q) E_0(Q) + \tfrac12\alpha^{-2}  |\xi|^2 [M(Q)]^2.
\end{align*}
Thus the proof of the corollary reduces to proving that 
\[
\forall \beta>0\ \exists (\alpha,\xi)\in\mathcal{P}: \alpha^{-2}|\xi|^2 = \beta. 
\]
To this end, let $\beta>0$. We take $\xi\in\R^3\backslash\{0\}$ satisfying
\[
|\xi|> \max\biggl\{ \biggl(\frac{\beta^{\frac12} C\|\nabla V\|_{L^4}}{\delta_0}\biggr)^{\frac12}, \biggl(\frac{\beta C\|\nabla V\|_{L^4}}{\lambda_1}\biggr)^{\frac12}, \beta^{\frac12}\biggr\},
\]
and define $\alpha := |\xi| \beta^{-\frac12}$. We need to show that $\exists \delta\in(0,\delta_0)$ such that $(\alpha,\xi)\in \mathcal{P}_\delta$.  

To this end, we select $\delta>0$ such that
\[
\frac{\beta^{\frac12}C\|\nabla V\|_{L^4}}{|\xi|^2}<\delta<\min\biggl\{\frac{\lambda_1}{\beta^{\frac12}},\delta_0\biggr\},
\]
which is possible due to the first two constraints on $|\xi|$.  Rearranging the lower bound on $\delta$ and using $\alpha=|\xi| \beta^{-\frac12}$ yields the first constraint in the definition of $\mathcal{P}_\delta$, while rearranging the upper bound on $\delta$ yields the second constraint.  Finally, the third constraint on $\xi$ guarantees $\alpha> 1$.  Thus we conclude that $(\alpha,\xi)\in\mathcal{P}_\delta$, which completes the proof of the corollary.
\end{proof}

We note that Landoulsi \cite{Landoulsi} has previously considered the analogue of Theorem~\ref{T} and Corollary~\ref{C} in the setting of the focusing cubic NLS in the presence of a convex obstacle (see also \cite{DLR, KVZ}).  Indeed, it was the work of Landoulsi that originally motivated us to consider the questions addressed here. We will discuss the work \cite{Landoulsi} in more detail after discussing the outline of the proof of Theorem~\ref{T}.

We would also like to mention the work of Cuccagna and Maeda \cite{CM1, CM2}, who carried out a similar construction in the setting of generalized mass-subcritical NLS.  In the mass-subcritical setting the underlying ground state is stable, and the primary challenges in \cite{CM1,CM2} lie in controlling internal modes and radiation damping via the nonlinear Fermi golden rule. In contrast, in our setting the ground state is unstable, and the key issue in our analysis is to control the unstable direction.  For this, we use a topological shooting argument in the style of \cite{CMM}.\footnote{We would also like to refer the reader to the work \cite{Schlag} on the construction of a stable manifold for \eqref{NLS0}, which addresses the issue of the unstable direction through a sophisticated contraction mapping approach.} We further manage the soliton-potential interaction through the explicit parameter constraints appearing in \eqref{Pdelta}.

\smallskip

\emph{Outline of the proof.} 

\smallskip

The proof is based on the strategy introduced in \cite{CMM} (which shares some connection with previous works such as \cite{Weinstein, MartelMerleTsai} and has been further adapted in works such as \cite{Landoulsi}). The solution in Theorem~\ref{T} is ultimately constructed as the limit of solutions $u_n$ that exhibit the correct asymptotic behavior on increasingly long intervals $[T_0,T_n]$.  Thus the essential step is the construction of a solution $u$ that exhibits the correct behavior on an interval $[T_0,T^*]$, where $T_0$ is large but fixed and $T^*$ may be arbitrarily large. 

Given parameters $(\alpha,\xi)\in\mathcal{P}_\delta$ and $T^*\gg 1$, we wish to prescribe data at $t=T^*$ to obtain a solution that is close to the traveling wave \eqref{W1} on $[T_0,T^*]$.  In fact, we will construct a solution that is close to a slightly \emph{modulated} traveling wave
\begin{equation}\label{W2}
W(t,x) =  e^{i[\alpha^2 t + x\cdot \xi - |\xi|^2 t + \gamma(t)]}Q_\alpha(x-2t\xi - y(t)),\quad Q_\alpha(x):=\alpha Q(\alpha x),
\end{equation}
where $(\gamma(t),y(t))\in\R^{1+3}$. Thus we look for a solution of the form
\begin{equation}\label{U1}
u(t,x)=e^{i[\alpha^2 t + x\cdot \xi - |\xi|^2 t + \gamma(t)]}\{Q_\alpha(x-2t\xi-y(t)) + z(t,x-2t\xi-y(t)\},
\end{equation}
and we need to control the perturbation $z(t)$ as well as the modulation parameters $(\gamma(t),y(t))$.

The heart of the argument is a bootstrap based on the Lyapunov functional
\begin{equation}\label{Lyapunov0}
H(u) := E_0(u) + \tfrac12(\alpha^2+|\xi|^2) M(u) - \tfrac12\xi\cdot P(u).
\end{equation}
Using the fact that $e^{ix\cdot\xi} Q_\alpha$ are critical points for $H$, one can compute that 
\begin{equation}\label{Lyapunov1}
H(u) - H(W_0) =\tfrac12\langle L_\alpha^+ z_1,z_1\rangle + \tfrac12 \langle L_\alpha^- z_2,z_2\rangle + \mathcal{O}(\|z\|_{H^1}^3),
\end{equation}
where $z=z_1+iz_2$ and $L_\alpha^\pm$ are the linearized operators defined in \eqref{Ls} below.  By prescribing data appropriately, we can arrange that $H(u)-H(W_0)$ is small at $t=T^*$. Thus, \emph{if} we had 
\begin{itemize}
\item[(i)] coercivity for the operators $L_\alpha^\pm$, and 
\item[(ii)] conservation of $H(u)-H(W_0)$, 
\end{itemize} 
then \eqref{Lyapunov1} would yield control over $z$ for all time by a continuity argument. 

As for item (i), the obstacles to coercivity for $L_\alpha^\pm$ are well-known (see e.g. Corollary~\ref{C:coercivity} below).  First, one must avoid the kernels of $L_\alpha^\pm$ (i.e. the span of $Q_\alpha$ and its spatial derivatives). This is precisely the role of the modulation parameters $(\gamma(t),y(t))$, whose variation can also be controlled as long as $z(t)$ remains small.  Second, one must control the components of $z(t)$ in the directions $Y^\pm_\alpha:=\alpha Y_1^\pm(\alpha\cdot)$. We return to this point below.

For item (ii), we first observe that since $W$ solves \eqref{NLS0} and $H(\cdot)$ is invariant under modulation, we have $H(W_0(t))\equiv H(W(t))\equiv H(W_0(T^*))$. However, because of the external potential, $H(u(t))$ is \emph{not} exactly conserved. Nonetheless, if $u$ has the form \eqref{U1} with $z$ sufficiently small and the soliton-potential interaction is sufficiently weak, there is hope that the continuity argument can still survive.

In light of the discussion above (and using Corollary~\ref{C:coercivity}), we may rewrite \eqref{Lyapunov1} in the form\footnote{We will be more careful about implicit constants below.}
\begin{align}
\|z(t)\|_{H^1}^2 & \lesssim |H(u(T^*))-H(W_0(T^*))|  \label{Lyapunov21} \\
& \quad + |\langle z(t), iQ_\alpha\rangle|^2 + \sum_{j=1}^3 |\langle z(t), \partial_{x_j} Q_\alpha\rangle|^2 \label{Lyapunov22} \\
& \quad + \int_t^{T*}\bigl| \tfrac{d}{ds} H(u(s))\bigr| \,ds + \sum_{\sigma\in\{\pm\}} |\langle iz(t),Y_\alpha^\sigma\rangle|^2 + \mathcal{O}(\|z(t)\|_{H^1}^3).  \label{Lyapunov23}
\end{align}

Now, the term \eqref{Lyapunov21} will be small by construction, while the terms in \eqref{Lyapunov22} will vanish identically by the choice of modulation parameters.  By direct calculation and using the decomposition \eqref{U1}, one finds that the integral term in \eqref{Lyapunov23} arising from non-conservation of $H$ is controlled schematically by
\begin{equation}\label{Lyapunov3}
\int_t^{T^*} \|\nabla V\|_{L^4} \|z(s)\|_{H^1}^2 + \|\nabla V(\cdot) \Psi(\cdot-2s\xi-y(s))\|_{L^2}\|z(s)\|_{H^1} \,ds,
\end{equation}
where $\Psi\in\{Q_\alpha,\nabla Q_\alpha\}$.  Because of the exponential decay of $Q$ and the decay assumption \eqref{V} on $V$, the term in \eqref{Lyapunov3} involving the soliton-potential interaction decays exponentially and can be readily incorporated into the bootstrap. The remaining integral term in \eqref{Lyapunov3} can also be incorporated into the bootstrap via Gronwall's inequality, provided the order of decay for $z$ is sufficiently large compared to the constant $\|\nabla V\|_{L^4}$. In fact, this is precisely what is guaranteed by first constraint on the traveling wave parameters appearing in \eqref{Pdelta}. 

It follows that the only possible obstruction to constructing our desired solution arises from the eigenfunction components
\[
A_\alpha^\pm(t):= \langle iz(t),Y_\alpha^\pm\rangle.
\]
One typically regards $Y_\alpha^+$ as the unstable direction and $Y_\alpha^-$ as the stable direction due to the signs of their corresponding eigenvalues; however, as we construct our solution backward in time, the roles are exactly reversed here. Using the equation satisfied by $z$ and the eigenvalue property, one can compute that
\begin{equation}\label{Adot}
\dot A_\alpha^\pm = \pm\alpha^2 \lambda_1 A_\alpha^\pm+ \mathcal{E}_\alpha^\pm,
\end{equation}
where $\mathcal{E}_\alpha^\pm$ is comprised of nonlinear contributions from $z$, modulation terms, and eigenfunction-potential interaction terms. In particular, under appropriate bootstrap assumptions, $\mathcal{E}_\alpha^\pm$ can be treated perturbatively.  Thus controlling $A_\alpha^+$ (backward in time) is straightforward, and we find that our construction can only break down due to the possible growth of the component $A_\alpha^-$ backward in time.

The final piece of the argument therefore consists of proving that we can prescribe data at time $t=T^*$ such that the component $A_\alpha^-$ remains under control on the entire interval $[T_0,T^*]$. This is done by a topological argument, as we now explain.

Writing 
\[
I^*=[-e^{-\frac12\delta\alpha|\xi|T^*},e^{-\frac12\delta\alpha|\xi|T^*}],
\]
the problem essentially boils down to proving that there exists $a\in I^*$ such that the solution we construct with $A_\alpha^-(T^*)=a$ satisfies
\begin{equation}\label{Abd}
\mathcal{N}(t;a) := |e^{\frac12\delta\alpha|\xi| t} A_\alpha^-(t)|^2\leq 1
\end{equation}
on the interval $[T_0,T^*]$.  Indeed, the continuity argument sketched above implies that if \eqref{Abd} holds on some interval of the form $[T_1, T^*]$, then our solution has the desired properties on this interval. We proceed by contradiction and assume that the construction fails before we reach $t=T_0$ for every choice of $a\in I^*$. In particular, after prescribing ${A}_\alpha^-(T^*)=a\in I^*$, we extend the solution backwards until some time  $\tau(a)\in(T_0,T^*]$ at which $\mathcal{N}(\tau(a),a)=1$.

It is at this point that we impose the second constraint appearing in \eqref{Pdelta}, which guarantees $-\alpha^2\lambda_1 + \tfrac12\alpha\delta|\xi|<0$.  Using $\eqref{Adot}$, this ultimately allows us to derive the following fact: \emph{$\mathcal{N}$ must be strictly decreasing at any time at which $\mathcal{N}$ takes the value $1$.}   This property allows us to prove that the map $a\mapsto \tau(a)$ is continuous via the Implicit Function Theorem, as well as the fact that $\tau(a) = T^*$ for $a\in \partial I^*$.

 We therefore find that the map $a\mapsto A_\alpha^-(\tau(a))$ is also continuous, and moreover it fixes the endpoints of $I^*$.  Thus by the Intermediate Value Theorem, there must exist $a^*\in I^*$ such that $A_\alpha^-(\tau(a^*))=0$.  In particular,  $\mathcal{N}(\tau(a^*),a^*)=0$, which contradicts that $\mathcal{N}(\tau(a),a)=1$ for all $a\in I^*$.

In conclusion, by prescribing $A_\alpha^-(T^*)$ properly and using the continuity argument described above, we can construct a solution that remains exponentially close to the modulated traveling wave on the entire interval $[T_0,T^*]$.  As the modulation parameters themselves are exponentially small, this yields a solution with the desired properties on $[T_0,T^*]$. To complete the proof of Theorem~\ref{T}, one must apply this construction on a sequence of final times $T_n\nearrow \infty$ and argue that the corresponding solutions $u_n$ converge to a solution that exhibits the correct asymptotic behavior as $t\to\infty$. 

\smallskip

\emph{Related work and followup questions.}  

\smallskip

As mentioned above, our construction parallels the one carried out by Landoulsi in the setting of the focusing NLS outside of a convex obstacle \cite{Landoulsi}.  In particular, \cite{Landoulsi} also adapts the techniques introduced in \cite{CMM}, relying on a bootstrap argument based on an appropriate Lyapunov functional.  In \cite{Landoulsi}, the analogue of the functional \eqref{Lyapunov0} uses the full conserved energy for the obstacle model (rather than the energy for \eqref{NLS0}), but still includes the (non-conserved) momentum.  However, we were unable to find any discussion in \cite{Landoulsi} addressing the failure of the conservation of momentum (and therefore the non-conservation of $H(u)$).  In particular, in the setting of \cite{Landoulsi} one should encounter integral terms over the boundary of the obstacle, and it should be necessary to account for terms analogous to those appearing in \eqref{Lyapunov3}.  We note that the treatment of \eqref{Lyapunov3} is a key new ingredient in this work, and controlling this term required that we restrict the traveling wave parameters, essentially to guarantee that the soliton-potential interaction is sufficiently weak compared to the size of the potential. 

We believe that the assumptions on $V$ in \eqref{V} could be relaxed considerably.  For example, while repulsivity and nonnegativity are natural for studying long-time dynamics of \eqref{NLS} (cf. Theorem~\ref{T:Hong}), they do not seem to be essential at all for our main construction.  All that is ultimately needed for the construction is a reasonable well-posedness theory for \eqref{NLS} and sufficient decay for the potential $V$.  The decay assumptions we impose in \eqref{V} are convenient but certainly not optimal. We also note that one could also allow the potential to be time-dependent, as long as the appropriate spatial decay assumptions hold uniformly in time.

However, one particularly interesting case that seems to be beyond the scope of the techniques presented here is that of the repulsive inverse-square potential, i.e. $V(x)=a|x|^{-2}$ with $a>0$.  For this model, the analogue of Theorem~\ref{T:Hong} holds (see \cite{MMZ, KMVZ}), but analogues of Theorem~\ref{T}  and Corollary~\ref{C} have not been established to date. We believe that this should be an interesting model for future study.

It is also natural to ask about the behavior of the solution constructed in Theorem~\ref{T} for times $t<T_0$. For a localized, nontrapping potential as in \eqref{V}, it seems that several scenarios could occur, including blowup, scattering as $t\to-\infty$, or convergence to another traveling wave as $t\to-\infty$.  We plan to consider this question in future work. 

\smallskip

\emph{Outline of the paper.}

\begin{itemize}
\item In Section~\ref{S:notation} we introduce notation and collect some preliminary results.  This includes a discussion of the linearized operator for \eqref{NLS0} around the ground state. 
 \item In Section~\ref{S:modulation} we prove that we can construct a local-in-time solution with prescribed components $A_\alpha^\pm$ in the directions $Y_\alpha^\pm$ at the final time, and that this solution admits a decomposition into a modulated traveling wave plus a remainder term satisfying certain orthogonality conditions (see Proposition~\ref{P:modulation}). 
\item In Section~\ref{S:bootstrap} we prove that the solution constructed in Proposition~\ref{P:modulation} satisfies improved estimates backward in time as long as the component $A_\alpha^-$ remains under control (see Proposition~\ref{P:bootstrap}). 
\item In Section~\ref{S:construction} we prove that we can prescribe the component $A_\alpha^-$ at the final time such that this component remains under control all the way back to the fixed time $T_0$, thus (by Proposition~\ref{P:bootstrap}) allowing us to construct a solution with the desired properties on arbitrarily long intervals (see Proposition~\ref{P:construction}).  
 \item Finally, in Section~\ref{S:proof} we prove the main result, Theorem~\ref{T}, by applying Proposition~\ref{P:construction} to a sequence of final times $T_n\to\infty$ and passing to the limit.
 \end{itemize}

\subsection*{Acknowledgements} L.~B. and J.~M. were supported in part by NSF grant DMS-2350225. J.~M. was additionally supported by Simons Foundation grant MPS-TSM-00006622. Part of this work was completed while X.~L. was visiting the University of Oregon as a Visiting Scholar under the support of NSFC grants No. 12371240 and No. 12431008.

\section{Notation and preliminaries}\label{S:notation}

We write $A\lesssim B$ to denote $A\leq CB$ for some $C>0$.  We will also occasionally use the `big-oh' notation $\mathcal{O}$.  We use the standard notation for Lebesgue norms, Sobolev norms, and mixed space-time norms. 

We use the following inner product throughout this paper: 
\[
\langle f,g\rangle := \Re \int f(x)\bar g(x)\,dx. 
\]

We record the following standard subcritical well-posedness result for \eqref{NLS} with potentials satisfying \eqref{V} (see e.g. \cite{Cazenave}).

\begin{proposition}\label{P:LWP} Suppose $V\in C^1(\R^3;\R)$ satisfies \eqref{V}, and fix $s\in(\frac12,1]$.

For any $t_0\in\R$ and $u_0\in H^s(\R^3)$, there exists an open interval $I\ni t_0$ and a unique solution $u\in C_t(I;H_x^s(\R^3))$ to \eqref{NLS} with $u(t_0)=u_0$.  Furthermore, the lifespan of the solution may be extended as long as $u(t)$ remains bounded in $H^s$.

Moreover, if $u_n:[T_1,T_2]\times\R^3\to\C$ are $H^s$-bounded solutions to \eqref{NLS} such that $u_n(t_0)\to u_0$ in $H^s$ for some $t_0\in[T_1,T_2]$ and $u$ is the solution to \eqref{NLS} with $u(t_0)=u_0$, then the maximal lifespan of $u$ contains $[T_1,T_2]$ and $u_n(t)\to u(t)$ strongly in $H^s$ for all $t\in [T_1,T_2]$. 
\end{proposition}

The proof of Proposition~\ref{P:LWP} relies on the standard arguments, namely, a contraction mapping argument based on the Duhamel formula for \eqref{NLS} and Strichartz estimates for $e^{-it(\Delta+V)}$.  We remark that the usual Strichartz estimates can readily be obtained under the assumptions \eqref{V} on the potential (see e.g. \cite{KT, GoldbergSchlag}).

We next record a simple technical lemma that will be used below to control soliton-potential and eigenfunction-potential interaction terms.

\begin{lemma}\label{L:WI} Suppose $\Psi_1,\Psi_2$ satisfy
\[
|\Psi_1(x)|+|\Psi_2(x)| \lesssim_c e^{-c|x|}.
\]
Then
\[
\|\Psi_1(\cdot)\Psi_2(\cdot-\eta)\|_{L^2} \lesssim_c e^{-\frac12 c|\eta|}. 
\]
\end{lemma}

\begin{proof} First, noting that $|x-\eta|\geq \tfrac12|\eta|$ if $|x|\leq \tfrac12|\eta|$, we estimate
\[
\int_{|x|\leq \frac12|\eta|} |\Psi_1(x)|^2 |\Psi_2(x-\eta)|^2\,dx  \lesssim_c e^{-c|\eta|} \int |\Psi_1(x)|^2\,dx \lesssim_c e^{-c|\eta|}.
\]
On the other hand,
\[
\int_{|x|>\frac12|\eta|}|\Psi_1(x)|^2 |\Psi_2(x-\eta)|^2\,dx \lesssim_c e^{-c|\eta|} \int |\Psi_2(x)|^2\,dx \lesssim_c e^{-c|\eta|},
\]
and thus we obtain the result. \end{proof}

We now introduce the operators $L_\alpha^\pm$ arising from the linearization of \eqref{NLS0} around the rescaled ground state $Q_\alpha(x):=\alpha Q(\alpha x)$: 
\begin{equation}\label{Ls}
L_\alpha^+ = \alpha^2 - \Delta - 3Q_\alpha^2,  \quad
L_\alpha^- =  \alpha^2 - \Delta - Q_\alpha^2.
\end{equation}

Writing $f=f_1+if_2$, we further introduce the operator $\mathcal{L}_\alpha$ as follows: 
\[
\mathcal{L}_\alpha f = i[L_\alpha^+ f_1 +i L_\alpha^- f_2] = i[-\Delta f + \alpha^2 f - 2 Q_\alpha^2 f - Q_\alpha^2 \bar f].
\]

Writing $D_\alpha$ for the scaling $[D_\alpha f](x) = f(\alpha x)$, we may write 
\begin{equation}\label{rescaling}
L_\alpha^\pm = \alpha^2 D_\alpha L_1^\pm D_\alpha^{-1}\qtq{and} \mathcal{L}_\alpha = \alpha^2 D_\alpha \mathcal{L}_1 D_\alpha^{-1}.
\end{equation}

We now recall some well-known properties of the operators $L_1^\pm$ and $\mathcal{L}_1$.  By rescaling, we will then derive the corresponding properties for $L_\alpha^\pm$ and $\mathcal{L}_\alpha$. 

First, we have the following (see e.g. \cite{Grillakis, WeinsteinM, DuyckaertsRoudenko, Landoulsi}).

\begin{lemma}\label{L:spectrum} The kernel of $L_1^+$ equals the span of $\{\partial_{x_j} Q\}_{j=1}^3$.  The kernel of $L_1^-$ equals the span of $Q$. 

The discrete spectrum of $\mathcal{L}_1$ is of the form $\{0,\pm\lambda_1\}$ for some $\lambda_1>0$.  The eigenvalues $\pm\lambda_1$ are simple, with $L^2$-normalized eigenfunctions $Y_1^\pm$ that decay exponentially (along with their derivatives). 

For $f=f_1+if_2\in H^1$, we have the following:
\begin{align*}
\|f\|_{H^1}^2 &\lesssim \langle L_1^+ f_1, f_1\rangle + \langle L_1^- f_2, f_2\rangle \\
& \quad + \sum_{j=1}^3 |\langle f,\partial_{x_j} Q\rangle|^2 + |\langle f, iQ\rangle|^2 + \sum_{\sigma\in\{\pm\}} |\langle if,Y_1^\sigma\rangle|^2.
\end{align*}
\end{lemma}

We can therefore obtain the following corollary.  We note that it is the coercivity estimate below that leads us to impose the restriction $\alpha>1$ in \eqref{Pdelta}. 

\begin{corollary}\label{C:coercivity} The kernel of $L_\alpha^+$ equals the span of $\{\partial_{x_j}Q_\alpha\}_{j=1}^3$.  The kernel of $L_\alpha^-$ equals the span of $Q_\alpha$.

The functions $Y^\pm_\alpha(x):=\alpha Y^\pm_1(\alpha x)$ are eigenfunctions of $\mathcal{L}_\alpha$ corresponding to eigenvalues $\pm\alpha^2\lambda_1$. 

Given $\alpha>1$, there exist $c>0$ (independent of $\alpha$) and $C(\alpha)>0$ such that for $f=f_1+if_2\in H^1$,
\begin{align*}
c\|f\|_{H^1}^2 & \leq \langle L_\alpha^+ f_1,f_1\rangle + \langle L_\alpha^- f_2,f_2\rangle \\
& \quad + C(\alpha)\biggl[\sum_{j=1}^3 |\langle f,\partial_{x_j} Q_\alpha\rangle|^2 + |\langle f, iQ_\alpha\rangle|^2 +\sum_{\sigma\in\{\pm\}} |\langle if,Y_\alpha^\sigma\rangle|^2\biggr]. 
\end{align*} 
\end{corollary}

\begin{proof} The claims about the spectra follow from \eqref{rescaling}.  For the coercivity estimate, we use Lemma~\ref{L:spectrum}, \eqref{rescaling}, and a change of variables to obtain the following: for $f=f_1+if_2$, 
\begin{align*}
\langle  L_\alpha^+ &f_1,f_1\rangle + \langle L_\alpha^- f_2,f_2\rangle \\
& = \alpha^{-1}\bigl[\langle L_1^+ D_\alpha^{-1} f_1,D_\alpha^{-1} f_1\rangle + \langle L_1^- D_\alpha^{-1} f_2,D_\alpha^{-1} f_2\rangle\bigr] \\
& \geq c\alpha^{-1}\|D_\alpha^{-1} f\|_{H^1}^2 \\
& \quad - C\alpha^{-1}\biggl[\sum_{j=1}^3 |\langle D_\alpha^{-1} f,\partial_{x_j}Q\rangle|^2 + |\langle D_\alpha^{-1}  f,iQ\rangle|^2 + \sum_{\sigma\in\{\pm\}} |\langle iD_\alpha^{-1} f, Y_1^\sigma\rangle|^2 \biggr] \\
& \geq c[\alpha^2\|f\|_{L^2}^2 + \|\nabla f\|_{L^2}^2] \\
& \quad - C(\alpha)\biggl[\sum_{j=1}^3 |\langle f,\partial_{x_j} Q_\alpha\rangle|^2 + |\langle f, iQ_\alpha\rangle|^2 +\sum_{\sigma\in\{\pm\}} |\langle if,Y_\alpha^\sigma\rangle|^2\biggr] \\
& \geq c\|f\|_{H^1}^2 - C(\alpha)\biggl[\sum_{j=1}^3 |\langle f,\partial_{x_j} Q_\alpha\rangle|^2 + |\langle f, iQ_\alpha\rangle|^2 + \sum_{\sigma\in\{\pm\}} |\langle if,Y_\alpha^\sigma\rangle|^2\biggr],
\end{align*}
where we have used $\alpha>1$ in the final inequality. 
\end{proof}

Finally, we observe that since $Y_\alpha^\pm$ are eigenfuctions of $\mathcal{L}_\alpha$ with nonzero eigenvalue and $\{Q_\alpha,i\partial_{x_1}Q_\alpha,i\partial_{x_2}Q_\alpha,i\partial_{x_3}Q_\alpha\}$ are in the kernel of the adjoint of $\mathcal{L}_\alpha$, we have the orthogonality relations
\begin{equation}\label{QY-orthog}
\langle Y_\alpha^\pm, Q_\alpha\rangle = \langle i Y_\alpha^\pm,\partial_{x_j} Q_\alpha\rangle = 0 \qtq{for}j\in\{1,2,3\}.
\end{equation}

\section{Modulated solution}\label{S:modulation} 

The goal of this section is to prove Proposition~\ref{P:modulation}, in which we construct a local-in-time solution with prescribed components $A_\alpha^\pm$ in the directions $Y_\alpha^\pm$ at a final time, and show that this solution admits a decomposition into a modulated traveling wave plus a remainder term satisfying certain orthogonality conditions.

Let us begin by introducing some notation, which will be used throughout the remainder of the paper.  Given $(\alpha,\xi)\in(1,\infty)\times\R^3\backslash\{0\}$ and $(\gamma,y)\in\R^{1+3}$, we define the time-dependent operators $\mathcal{T}_{(t;\alpha,\xi,\gamma,y)}$ by
\[
[\mathcal{T}_{(t;\alpha,\xi,\gamma,y)}f](x) = e^{i[\alpha^2 t + x\cdot\xi - |\xi|^2 t +\gamma]}f(x-2t\xi-y). 
\]

Writing
\[
Q_\alpha(x):=\alpha Q(\alpha x),
\]
the traveling wave $W$ with parameters $(\alpha,\xi)$ defined by \eqref{W1} corresponds to
\[
W_0(t)= \mathcal{T}_{(t;\alpha,\xi,0,0)} Q_\alpha,
\]
while the modulated traveling wave $\tilde W$ corresponding to parameters $(\alpha,\xi,\gamma(t),y(t))$ is defined by 
\[
W(t) = \mathcal{T}_{(t;\alpha,\xi,\gamma(t),y(t))} Q_\alpha.
\]

We first have the following modulation result at fixed times. 
\begin{lemma}\label{L:modulation} Fix $(\alpha,\xi)\in(1,\infty)\times\R^3\backslash\{0\}$, $t\in\R$, and $f\in H^1$. There exists $\eps_0$ such that if
\begin{equation}\label{modulation-close}
\|f-\mathcal{T}_{(t;\alpha,\xi,0,0)}Q_\alpha\|_{H^1}<\eps_0,
\end{equation}
then there exist $(\gamma(t),y(t))\in\R^{1+3}$ such that the following holds: if we define $z(t)\in H^1$ by
\begin{equation}\label{modulation-z}
f = \mathcal{T}_{(t;\alpha,\xi,\gamma(t),y(t))}[Q_\alpha + z(t)], 
\end{equation}
then we have
\begin{equation}\label{modulation-orthog}
\langle iz(t), Q_\alpha\rangle = \langle z(t),\partial_{x_j} Q_\alpha\rangle   = 0 \qtq{for}j\in\{1,2,3\}
\end{equation}
and
\begin{equation}\label{modulation-bds}
\|z(t)\|_{H^1}+|\gamma(t)|+|y(t)| \lesssim \|f-\mathcal{T}_{(t;\alpha,\xi,0,0)}Q_\alpha\|_{H^1}.
\end{equation}
\end{lemma}

\begin{proof} As the argument follows along standard lines, we will keep our presentation brief.  We define $\Phi:H^1\times\R^{1+3}\to\R^{1+3}$ by
\begin{align*}
\Phi_0(g,(\gamma,y)) & = \langle i\{\mathcal{T}_{(t;\alpha,\xi,\gamma,y)}^{-1}[g + \mathcal{T}_{(t;\alpha,\xi,0,0)}Q_\alpha]-Q_\alpha\},Q_\alpha\rangle, \\
\Phi_j(g,(\gamma,y)) &= \langle \mathcal{T}_{(t;\alpha,\xi,\gamma,y)}^{-1}[g + \mathcal{T}_{(t;\alpha,\xi,0,0)}Q_\alpha]-Q_\alpha,\partial_{x_j} Q_\alpha\rangle,\quad j\in\{1,2,3\}.
\end{align*}
We will use the Implicit Function Theorem to show that for $g$ in an $H^1$-neighborhood of $0$, there exist unique $(\gamma,y)$ such that $\Phi(g,(\gamma,y))=0$. 

To this end, first observe that $\Phi(0,(0,0))=(0,0)$.  Next, by direct calculation we can obtain 
\[
D_{(\gamma,y)} \Phi\big|_{(0,(0,0))} = \left[\begin{array}{cccc}  \|Q_\alpha\|_{L^2}^2 & 0 & 0 & 0 \\ 0 &  \|\partial_{x_1} Q_\alpha\|_{L^2}^2 & 0 & 0  \\ 0 & 0 & \|\partial_{x_2} Q_\alpha\|_{L^2}^2 & 0 \\ 0 & 0 & 0 & \|\partial_{x_3} Q_\alpha\|_{L^2}^2 \end{array}\right],
\]
so that $|D_{(\gamma,y)}\Phi|_{(0,(0,0))}|\gtrsim_\alpha 1$.

Thus, by the Implicit Function Theorem, there exist open sets  $U\subset H^1$ and $V\subset \R^{1+3}$  with $(0,(0,0))\in U\times V$ and $h\in C^1(U;V)$ such that
\[
\Phi(g,(\gamma,y))=0 \qtq{if and only if} (\gamma,y)=h(g). 
\]

Now observe that for $\eps_0$ sufficiently small, \eqref{modulation-close} guarantees that 
\[
g(t):=f-\mathcal{T}_{(t;\alpha,\xi,0,0)}Q_\alpha\in U.
\] 
We then set $(\gamma(t),y(t)) = h(g(t))$ and define $z(t)$ as in \eqref{modulation-z}.  To complete the proof, we must verify \eqref{modulation-orthog} and \eqref{modulation-bds}.  To this end, we first observe that by construction, \eqref{modulation-orthog} is exactly equivalent to $\Phi(g(t),(\gamma(t),y(t))) = 0$.  As for \eqref{modulation-bds}, we first observe that
\[
|\gamma(t)|+|y(t)| \lesssim \|g(t)\|_{H^1}= \|f-\mathcal{T}_{(t;\alpha,\xi,0,0)}Q_\alpha\|_{H^1}.
\]
Finally, by construction and direct calculation
\begin{align*}
\|z(t)\|_{H^1} & \leq \| \mathcal{T}_{(t;\alpha,\xi,\gamma(t),y(t))}^{-1}[g(t) + \mathcal{T}_{(t;\alpha,\xi,0,0)} Q_\alpha] - Q_\alpha\|_{H^1}  \\
& \leq \| \mathcal{T}^{-1}_{(t;\alpha,\xi,\gamma(t),y(t))} g(t)\|_{H^1} + \|\mathcal{T}^{-1}_{(t;\alpha,\xi,\gamma(t),y(t))}\mathcal{T}_{(t;\alpha,\xi,0,0)} Q_\alpha - Q_\alpha\|_{H^1} \\
& \lesssim \|g(t)\|_{H^1} + |\gamma(t)|+|y(t)| \\
& \lesssim \|f-\mathcal{T}_{(t;\alpha,\xi,0,0)}Q_\alpha\|_{H^1},
\end{align*}
which completes the proof. \end{proof}

We will apply Lemma~\ref{L:modulation} for $t$ belonging to intervals of the form $I=[T,T^*]$, obtaining functions of the form $\gamma:I\to\R$ and $y:I\to\R^3$.

\begin{proposition}\label{P:modulation} Let $\delta\in(0,\delta_0)$ and $(\alpha,\xi)\in \mathcal{P}_\delta$. There exists $\{C_j\}_{j=0}^2\subset(0,\infty)$ such that for all $T^*$ sufficiently large the following holds.  

Let 
\[
a\in I^*:= [-e^{-\frac12\delta\alpha|\xi| T^*}, e^{-\frac12\delta\alpha|\xi| T^*}].
\]
There exist $\beta^\pm \in[-C_0a,C_0a]$ and $T<T^*$ such that if
\begin{equation}\label{final-data}
\varphi= \mathcal{T}_{(T^*;\alpha,\xi,0,0)}\biggl[Q_\alpha - \sum_{\sigma\in\{\pm\}} i\beta^\sigma Y_\alpha^\sigma\biggr],
\end{equation}
then the solution $u$ to \eqref{NLS} with $u|_{t=T^*}=\varphi$ exists on $[T,T^*]$ and there exist $C^1$ functions $\gamma:[T,T^*]\to\R$ and $y:[T,T^*]\to\R^3$ with
\begin{equation}\label{finalstate0}
(\gamma(T^*),y(T^*))=(0,0)
\end{equation}
such that the following hold for all $t\in[T,T^*]$: 
\begin{align}
\|u(t)-\mathcal{T}_{(t;\alpha,\xi,0,0)}Q_\alpha\|_{H^1} < \eps_0,\label{bound0}
\end{align}
and if $z,A_\alpha^\pm$ are defined via
\begin{equation}\label{zA}
u(t) = \mathcal{T}_{(t;\alpha,\xi,\gamma(t),y(t))}[Q_\alpha + z(t)]\qtq{and} A_\alpha^\pm(t)=\langle iz(t), Y_\alpha^\pm\rangle,
\end{equation}
then
\begin{align}
(A_\alpha^+(T^*),A_\alpha^-(T^*))&=(0,a), \label{finalstate0a} \\
\langle z(t),\partial_{x_j} Q_\alpha\rangle = \langle iz(t),Q_\alpha\rangle&=0\qtq{for} j\in\{1,2,3\}, \label{orthogonality}
\end{align}
and we have the following bounds: 
\begin{align}
\|z(t)\|_{H^1} & \leq C_1 e^{-\frac13\delta\alpha|\xi|t}, \label{bound1}\\
|\gamma(t)| + |y(t)| & \leq C_2 e^{-\frac13\delta\alpha|\xi|t}, \label{bound2} \\
|A_\alpha^+(t)| & \leq e^{-\frac12\delta\alpha|\xi|t}. \label{bound3}
\end{align}
Furthermore, the solution may be extended backward in time as long as \eqref{bound0} holds, and in this case we can also extend $(\gamma(t),y(t))$ such that \eqref{orthogonality} holds. 
\end{proposition}

\begin{proof}[Proof of Proposition~\ref{P:modulation}]  First observe that for $a\in I^*$ and \emph{any} $\beta^\pm\in[-C_0 a,C_0 a]$, if $\varphi$ is as in \eqref{final-data} then
\begin{align*}
\|\varphi - \mathcal{T}_{(T^*;\alpha,\xi,0,0)}Q_\alpha\|_{H^1} & \leq C_0 e^{-\frac12\delta\alpha|\xi|T^*} \sum_{\sigma\in\{\pm\}} \|\mathcal{T}_{(T^*;\alpha,\xi,0,0)} Y_\alpha^\sigma\|_{H^1}  \\
&\lesssim e^{-\frac12\delta\alpha|\xi|T^*}.
\end{align*}
Thus, applying the local theory for \eqref{NLS} and choosing $T^*$ sufficiently large, we find that the solution $u$ to \eqref{NLS} with $u|_{t=T^*}=\varphi$ exists and satisfies
\[
\|u(t)-\mathcal{T}_{(t;\alpha,\xi,0,0)}Q_{\alpha}\|_{H^1} \lesssim e^{-\frac13\delta|\alpha| t} <\eps_0 
\]
on some interval $[T,T^*]$ with $T<T^*$. 

It follows that we may apply Lemma~\ref{L:modulation} with $f=u(t)$ on the interval $[T,T^*]$.  In particular we may define modulation parameters $(\gamma(t),y(t))$ and $z(t)$ as in \eqref{zA} such that the orthogonality conditions \eqref{orthogonality} hold.  The fact that $(\gamma(T^*),y(T^*))=0$ follows from \eqref{final-data}, the uniqueness in the Implicit Function Theorem, and the orthogonality conditions \eqref{QY-orthog}.  Moreover, as long as \eqref{bound0} holds, we obtain $H^1$-bounds for $u$ and the modulation constraint \eqref{modulation-close}, which allows us to extend the solution $u(t)$ and as well as the functions $(\gamma(t),y(t))$.

Lemma~\ref{L:modulation} additionally yields the bounds 
\[
\|z(t)\|_{H^1} + |y(t)|+|\gamma(t)| \lesssim \|u(t)-\mathcal{T}_{(t;\alpha,\xi,0,0)}Q_{\alpha}\|_{H^1} \lesssim e^{-\frac13\delta|\alpha| t},
\] 
which imply \eqref{bound1}--\eqref{bound2}.  We also remark that the Implicit Function Theorem may be used to derive that $\gamma,y\in C_t^1([T,T^*])$. 
  
We now show that given $a\in I^*$ we can find unique $\beta^\pm\in[-C_0a,C_0a]$ to impose \eqref{finalstate0a}. To this end let us define the {invertible} matrix $\Upsilon_\alpha$ by
\[
\Upsilon_\alpha = \left[\begin{array}{cc} \langle Y_\alpha^+,Y_\alpha^+\rangle & \langle Y_\alpha^-,Y_\alpha^+\rangle \\ \langle Y_\alpha^+,Y_\alpha^-\rangle & \langle Y_\alpha^-,Y_\alpha^-\rangle \end{array}\right]\qtq{and let} \left[\begin{array}{c} \beta^+ \\ \beta^- \end{array} \right] = \Upsilon_\alpha^{-1} \left[\begin{array}{c} 0 \\ a \end{array}\right].
\]
It follows that $|\beta^\pm| \lesssim_\alpha a$, and \eqref{finalstate0a} follows directly from \eqref{final-data} and \eqref{finalstate0}.

Finally, as $A_\alpha^+(T^*)=0$, we can guarantee that \eqref{bound3} holds by shrinking the interval $[T,T^*]$ if necessary.\end{proof}

\section{Bootstrap argument}\label{S:bootstrap}

Given $\delta\in(0,\delta_0)$ and $(\alpha,\xi)\in\mathcal{P}_\delta$, we may choose $T_0$ sufficiently large that
\[
e^{-\frac{1}{10}\delta\alpha|\xi| T_0} \ll \eps_0,
\]
where $\eps_0$ is the modulation threshold as in Lemma~\ref{L:modulation}.  We assume that $T_0$ is chosen in this fashion throughout the remainder of the paper. 

The next proposition shows that we can improve our estimates on the solution constructed in Proposition~\ref{P:modulation} as long as the component $A^-(t)$ remains under control.

\begin{proposition}\label{P:bootstrap} Let $\delta\in(0,\delta_0)$ and $(\alpha,\xi)\in \mathcal{P}_\delta$. Importing all notation from Proposition~\ref{P:modulation}, let 
\[
a\in [-e^{-\frac12\delta\alpha|\xi| T^*}, e^{-\frac12\delta\alpha|\xi| T^*}],
\]
and let $u:[T,T^*]\times\R^3\to\C$ be the solution constructed in Proposition~\ref{P:modulation} with
\[
(A_\alpha^+(T^*),A_\alpha^-(T^*))=(0,a).
\]
 Assume $T\geq T_0$.  If
\begin{equation}\label{bound4}
|A^-(t)|\leq e^{-\frac12\delta\alpha|\xi|t}
\end{equation}
for all $t\in[T,T^*]$, then we have
\begin{align}
\|z(t)\|_{H^1} & \lesssim e^{-\frac12\delta\alpha|\xi|t}, \label{bound12}\\
|y(t)| + |\gamma(t)| & \lesssim e^{-\frac12\delta\alpha|\xi|t}, \label{bound22} \\
|A_\alpha^+(t)| & \lesssim e^{-\frac23\delta\alpha|\xi|t}\label{bound32}
\end{align}
for all $t\in[T,T^*]$.  In particular,
\begin{equation}\label{mainbd}
\|u(t) - \mathcal{T}_{(t;\alpha,\xi,0,0)}Q_\alpha\|_{H^1} \lesssim e^{-\frac12\delta\alpha|\xi|t}
\end{equation}
for $t\in[T,T^*]$. 
\end{proposition}

\begin{proof} We first prove \eqref{bound12}.  We define 
\[
H(u) := E_0(u) + \tfrac12(\alpha^2+|\xi|^2) M(u) - \tfrac12\xi\cdot P(u).
\]
Let us denote the traveling wave and modulated traveling wave  by
\[
W_0(t) = \mathcal{T}_{(t;\alpha,\xi,0,0)} Q_\alpha \qtq{and} W(t)= \mathcal{T}_{(t;\alpha,\xi,\gamma(t),y(t))} Q_\alpha,
\]
and recall that we have
\[
u(t)=\mathcal{T}_{(t;\alpha,\xi,\gamma(t),y(t))}[Q_\alpha+z(t)].
\]

As $W_0$ solves \eqref{NLS0} and $H$ is gauge- and translation-invariant, we have 
\begin{equation}\label{HW}
H(W_0(t))\equiv H( W(t))\equiv H(W_0(T^*)),
\end{equation}
while by explicit calculation we obtain
\begin{align}
H(u) - H(W) & = \tfrac12\int |\nabla z|^2\,dx + \tfrac12\alpha^2\int |z|^2\,dx  - \tfrac12\int Q_\alpha^2[3z_1^2 + z_2^2]\,dx  \label{Lappears} \\
& \quad + \int [\alpha^2 Q_\alpha -\Delta Q_\alpha - Q_\alpha^3]z_1 \,dx \label{Qaeq} \\
& \quad -\int Q_\alpha z_1 |z|^2\,dx - \tfrac14\int |z|^4\,dx,\label{HOT}
\end{align}
where $z=z_1+iz_2$.  We now use the operators $L_\alpha^\pm$ defined in \eqref{Ls} to rewrite \eqref{Lappears}, and observe that by \eqref{Qeq} we have $\eqref{Qaeq}\equiv 0$.  Thus, using \eqref{bound1} as well, the identity above may be rewritten
\begin{equation}\label{HuHW1}
H(u) - H(W) = \tfrac12\langle L_\alpha^+ z_1,z_1\rangle + \tfrac12\langle L_\alpha^- z_2,z_2\rangle + \mathcal{O}_\alpha(\|z\|_{H^1}^3). 
\end{equation}

Now, $H(u(t))$ is \emph{not} exactly conserved. Indeed, using \eqref{NLS} and the decomposition of $u$, we compute
\begin{align*}
\tfrac{d}{dt} H(u(t))  & = -\Im\int \bar u \nabla V \cdot \nabla u \,dx + \xi\cdot\int \nabla V\, |u|^2\,dx  \\
& = -\Im \int \nabla V(x+2\xi t + y(t)) \cdot\bigl\{Q_\alpha \nabla z + \bar z \nabla Q_\alpha+ \bar z\nabla z\bigr\}\,dx. 
\end{align*}

Applying the Fundamental Theorem of Calculus, \eqref{HW}, and Corollary~\ref{C:coercivity}, we therefore obtain
\begin{align}
c\|z(t)\|_{H^1}^2 & \leq |H(u(T^*))-H(W_0(T^*))| \label{bs1}\\ 
& \quad + C(\alpha)\biggl[ |\langle z_2, iQ_\alpha\rangle|^2 + \sum_{j=1}^3 |\langle z_1, \partial_{x_j} Q_\alpha\rangle|^2 \biggr]\label{bs2} \\
& \quad + C\sum_{\Psi\in S} \int_t^{T^*} \| \nabla V \,\Psi(\cdot-\eta(s))\|_{L^2} \|z(s)\|_{H^1} \label{bs3} \\
& \quad + C\int_t^{T^*} \|\nabla V\|_{L^4} \|z(s)\|_{H^1}^2\,ds \label{bs4}\\
& \quad + C(\alpha)\biggl[ \sum_{\sigma\in\{\pm\}} |\langle iz(t),Y_\alpha^\sigma\rangle|^2+ \|z(t)\|_{H^1}^3\biggr],\label{bs5}
\end{align}
where $S=\{Q_\alpha,\partial_{x_1}Q_\alpha,\partial_{x_2}Q_\alpha,\partial_{x_3}Q_\alpha\}$ and 
\[
\eta(s):=2s\xi + y(s).
\]

Now let us utilize the estimates \eqref{bound1}--\eqref{bound3} and \eqref{bound4} that we have on $[T,T^*]$ to estimate \eqref{bs1}--\eqref{bs5}.

First, using \eqref{HuHW1}, the definition of $u(T^*)$ in \eqref{final-data}, and \eqref{finalstate0}, we have
\[
|\eqref{bs1}| \lesssim \|z(T^*)\|_{H^1}^2 \lesssim |\beta^+|^2+|\beta^-|^2 \lesssim e^{-\delta\alpha|\xi| t}. 
\]

Next, by the orthogonality conditions \eqref{orthogonality}, we have $\eqref{bs2}\equiv 0$. 

For \eqref{bs3}, we will use the decay assumption \eqref{V} and apply Lemma~\ref{L:WI} with $c=\delta\alpha$ and $\eta=\eta(s)$. Note that by using \eqref{bound2} and $s\geq T_0\geq 1$, we have
\begin{equation}\label{ApplyWI}
\tfrac12\delta\alpha|\eta(s)| \geq \delta\alpha s|\xi|-\tfrac12\delta\alpha|y(s)|\geq \tfrac23 \delta\alpha |\xi| s. 
\end{equation} 
Thus, using \eqref{bound1} as well, we find
\[
|\eqref{bs3}| \lesssim \int_t^{T^*} e^{-\delta\alpha |\xi| s}\,ds \lesssim e^{-\delta\alpha|\xi|t}. 
\]

The term \eqref{bs4} will be treated via using Gronwall's inequality. 

For \eqref{bs5}, we use \eqref{bound1}, \eqref{bound3}, and \eqref{bound4} to obtain
\[
|\eqref{bs5}| \lesssim e^{-\delta\alpha|\xi| t}. 
\] 

Combining all of the estimates above, we obtain
\[
\|z(t)\|_{H^1}^2 \leq C(\alpha) e^{-\delta\alpha|\xi|t} + \int_t^{T^*} C\|\nabla V\|_{L^4} \|z(s)\|_{H^1}^2\,ds.
\]
Noting that the first constraint in \eqref{Pdelta} guarantees
\[
\delta\alpha|\xi|>C\|\nabla V\|_{L^4},
\]
the bound \eqref{bound12} now follows from an application of Gronwall's inequality. 

We turn to the proof of the estimate \eqref{bound22} on the modulation parameters. In light of \eqref{finalstate0}, it suffices to prove that
\begin{equation}\label{bound22-pf}
|\dot\gamma(t)| + |\dot y(t)| \lesssim e^{-\frac12\delta\alpha|\xi|t}. 
\end{equation}
To this end, we first observe that direct calculation using \eqref{NLS} and \eqref{Qeq} leads to the following evolution equation for $z$:
\begin{equation}\label{zeq}
\begin{aligned}
i\partial_t z &= L_\alpha^+ z_1 + iL_\alpha^- z_2 +  V(\cdot+2t\xi+y(t))(Q_\alpha+z) \\
& \quad + \dot\gamma(Q_\alpha+z) +i\dot y\cdot\nabla(Q_\alpha+z)  - [2Q_\alpha |z|^2 + Q_\alpha z^2 + |z|^2 z].
\end{aligned}
\end{equation}
Thus, differentiating the relations \eqref{orthogonality}, we obtain
\begin{equation}\label{DO}
\langle \Re \text{RHS}\eqref{zeq},Q_\alpha\rangle = \langle \Im \text{RHS}\eqref{zeq},\partial_{x_j} Q_\alpha\rangle = 0,\quad j\in\{1,2,3\}. 
\end{equation}
We now rewrite the first equation in \eqref{DO} as 
\begin{align*}
\dot\gamma  \langle Q_\alpha+z&,Q_\alpha\rangle - \dot y \cdot\langle \nabla z_2,Q_\alpha\rangle \\
 & = -\langle z_1, L_\alpha^+ Q_\alpha\rangle - \langle V(\cdot+2t\xi + y(t))(Q_\alpha+z_1),Q_\alpha\rangle \\
& \quad  + \langle 2 Q_\alpha |z|^2 + Q_\alpha\Re(z^2) + |z|^2 z_1,Q_\alpha\rangle,
\end{align*}
and fixing $j\in\{1,2,3\}$, we rewrite the second equation in \eqref{DO} as
\begin{align*}
\dot y \cdot  \langle\nabla&(Q_\alpha+z_1),\partial_{x_j} Q_\alpha\rangle  + \dot\gamma\langle z_2,\partial_{x_j} Q_\alpha\rangle \\
& = -\langle z_2, L_\alpha^- \partial_{x_j} Q_\alpha\rangle - \langle V(\cdot+2t\xi+y(t))z_2,\partial_{x_j} Q_\alpha\rangle \\
& \quad + \langle Q_\alpha\Im(z^2) + |z|^2 z_2,\partial_{x_j} Q_\alpha\rangle. 
\end{align*}
Writing $\pi = [ \gamma\ y ]^t\in \R^{1+3}$, we can combine these four equations into a system of the form
\[
X_1 \dot \pi + X_2 \dot \pi = X_3,
\]
where
\[
X_1 := \text{diag}\{\langle Q_\alpha,Q_\alpha\rangle, \langle \partial_{x_1}Q_\alpha,\partial_{x_1}Q_\alpha\rangle, \langle \partial_{x_2}Q_\alpha,\partial_{x_2}Q_\alpha\rangle, \langle \partial_{x_3}Q_\alpha,\partial_{x_3}Q_\alpha\rangle\} 
\]
and we have the bounds
\begin{align*}
\|X_2\| & \lesssim \|z(t)\|_{H^1}, \\
\|X_3\| &\lesssim \|z(t)\|_{H^1} + \|V(\cdot + 2t\xi + y(t))Q_\alpha\|_{L^2}\bigl\{\|Q_\alpha\|_{L^2}+\|z\|_{L^2}\bigr\}  \\
& \quad + \|z(t)\|_{H^1}^2+\|z(t)\|_{H^1}^3.
\end{align*}
In particular, using \eqref{bound12}, we obtain 
\begin{equation}\label{pidotbd}
|\dot \gamma(t)|+|\dot y(t)| \lesssim \sum_{j=1}^3 \|z(t)\|_{H^1}^j +  \|V(\cdot + 2t\xi + y(t))Q_\alpha\|_{L^2},
\end{equation}
which by \eqref{bound12} and Lemma~\ref{L:WI} (as applied above) yields \eqref{bound22-pf} and hence \eqref{bound22}. 

It remains to prove \eqref{bound32}.  To this end, we first use the definition \eqref{zA} and \eqref{zeq} to obtain
\begin{align*}
\dot A_\alpha^+ & = \langle \text{RHS}\eqref{zeq},Y_\alpha^+\rangle=:\langle L_\alpha^+ z_1 + i L_\alpha^- z_2, Y_\alpha^+\rangle + \mathcal{E}_\alpha^+,
\end{align*}
where explicitly we have
\begin{equation}\label{Eplus}
\begin{aligned}
\mathcal{E}_\alpha^+ & = \langle V(\cdot+2t\xi+y(t))(Q_\alpha+z),Y_\alpha^+\rangle + \dot\gamma\langle z,Y_\alpha^+\rangle + \dot y \cdot\langle i\nabla z,Y_\alpha^+\rangle\\
& \quad - \langle 2Q_\alpha|z|^2+Q_\alpha z^2 + |z|^2 z,Y_\alpha^+\rangle.
\end{aligned}
\end{equation}
We note that the terms $\dot\gamma\langle Q_\alpha,Y_\alpha^+\rangle$ and $\dot y\cdot \langle i\nabla Q_\alpha,Y_\alpha^+\rangle$ in $\mathcal{E}_\alpha^+$ vanish identically due to the orthogonality relations \eqref{QY-orthog}. 

We now observe that
\[
\langle L_\alpha^+ z_1 + i L_\alpha^- z_2,Y_\alpha^+\rangle = \langle iz,\mathcal{L}_\alpha Y_\alpha^+\rangle = \alpha^2\lambda_1 A_\alpha^+(t), 
\]
while estimating as above (using \eqref{bound12}, \eqref{bound22-pf}, and Lemma~\ref{L:WI} with \eqref{ApplyWI}) leads to
\begin{align*}
|\mathcal{E}_\alpha^+(t)| & \lesssim \{|\dot\gamma(t)|+|\dot y(t)|\}\|z(t)\|_{H^1}\|Y_\alpha^+\|_{L^2} \\
& \quad + \|V(\cdot + 2t\xi + y(t))Y_\alpha^+\|_{L^2}[\|Q_{\alpha}\|_{L^2}+\|z(t)\|_{L^2}] \\
& \quad +   \|z(t)\|_{H^1}^2 + \|z(t)\|_{H^1}^3 \\
& \lesssim e^{-\delta\alpha|\xi|t}+ e^{-\frac23\delta\alpha|\xi|t} 
\end{align*}
Thus 
\[
\dot A_\alpha^+(t) - \alpha^2 \lambda_1 A_\alpha^+(t) = \mathcal{O}(e^{-\frac23\delta\alpha|\xi|t}),
\]
which (recalling $A_\alpha^+(T^*)=0$) integrates to give \eqref{bound32}.\end{proof}
 
\section{Main construction}\label{S:construction}

We turn to the main construction in this paper, namely, a solution to \eqref{NLS} that remains close to a traveling wave on an arbitrarily long interval of $[T_0, T^*]$. 

\begin{proposition}\label{P:construction} Let $\delta\in(0,\delta_0)$ and $(\alpha,\xi)\in \mathcal{P}_\delta$.  Importing all notation from Proposition~\ref{P:modulation}, there exists 
\[
a\in I^*:=[-e^{-\frac12\delta\alpha|\xi|T^*},e^{-\frac12\delta\alpha|\xi|T^*}]
\]
such that the solution constructed in Proposition~\ref{P:modulation} with
\begin{equation}\label{finalstate0a2}
(A_\alpha^+(T^*),A_\alpha^-(T^*))=(0,a)
\end{equation}
exists on $[T_0,T^*]$ satisfies \eqref{bound4}--\eqref{bound32}  on this interval. 
\end{proposition}

\begin{proof} Given $a\in I^*$, we consider the solution $u$ constructed in Proposition~\ref{P:modulation} satisfying \eqref{finalstate0a2}. Let us call an interval $I$ of the form $[T,T^*]$ \emph{stable for $u$} if the bounds \eqref{bound0}, \eqref{bound1}--\eqref{bound3}, and \eqref{bound4} hold for all $t\in I$.  As observed in Proposition~\ref{P:modulation}, the condition \eqref{bound0} on $I$ is enough to guarantee that the solution $u$ and the modulation parameters imposing the orthogonality conditions can be defined on an open interval containing $I$.  Let us also remark that \eqref{bound1}--\eqref{bound2} themselves imply \eqref{bound0} provided $t\geq T_0$.    

To construct the desired solution, it suffices to show that there exists $a\in I^*$ such that $[T_0,T^*]$ is stable for $u$.  Indeed, in this case Proposition~\ref{P:bootstrap} allows us to improve \eqref{bound1}--\eqref{bound3} to the stronger estimates \eqref{bound12}--\eqref{bound32}. To this end, we define 
\[
\tau(a) := \inf\{T\geq T_0: [T,T^*]\text{ is stable for }u\}
\]
and let us suppose towards a contradiction that $\tau(a)>T_0$ for all $a\in I^*$. 

We now introduce the function 
\[
\mathcal{N}(t;a) := |e^{\frac12\delta\alpha|\xi| t} A_\alpha^-(t)|^2,
\]
which we may define on some open interval containing $[\tau(a),T^*]$. 

By continuity of the flow and Proposition~\ref{P:modulation}, for any $a\in (I^*)^\circ$ we have $\tau(a)<T^*$.  Moreover, if $I=[T,T^*]$ is stable for $u$, then Proposition~\ref{P:bootstrap} implies the improved bounds \eqref{bound12}--\eqref{bound32}, so that (using $T_0\gg 1$) we have
\begin{align*}
\|z(t)\|_{H^1} & \lesssim e^{-\frac12\delta\alpha|\xi| t} \leq\tfrac12 C_1 e^{-\frac13\delta\alpha|\xi| t}, \\
|\gamma(t)|+|y(t)| & \lesssim e^{-\frac12\delta\alpha|\xi|t} \leq \tfrac12 C_2 e^{-\frac13\delta\alpha |\xi| t}, \\
|A_\alpha^+(t)| &\lesssim e^{-\delta\alpha|\xi| t}\leq \tfrac12 e^{-\frac12\delta\alpha|\xi| t}
\end{align*}
for all $t\in I$. It follows that the failure of stability can only be due to the failure of \eqref{bound4}; more precisely, given $a\in (I^*)^\circ$ we have that
\begin{equation}\label{Ntaua}
\mathcal{N}(\tau(a);a)=1.
\end{equation}
We will see below that $\tau(a)=T^*$ for $a\in\partial I^*$, so that \eqref{Ntaua} in fact holds for all $a\in I^*$. 

We now claim that the following differential inequality holds for $\mathcal{N}$: if $t\geq T_0$ is a time at which the modulation constraint \eqref{bound1} and the improved bound \eqref{bound12} on $z$ hold, then 
\begin{equation}\label{DIN}
\begin{aligned}
\partial_t \mathcal{N}(t;a) & \leq -\bigl\{2\alpha^2 \lambda_1 - \delta\alpha|\xi|\bigr\}\mathcal{N}(t;a) + C(\alpha)e^{-\frac16\delta\alpha|\xi|t}\sqrt{N(t;a)}.
\end{aligned}
\end{equation}
To prove \eqref{DIN}, we first introduce the quantity $\mathcal{E}_\alpha^-(t)$ (the analogue of the quantity $\mathcal{E}_\alpha^+$ defined in \eqref{Eplus}), which (proceeding as in the proof of Proposition~\ref{P:bootstrap}) satisfies
\[
\dot A_\alpha^- = -\alpha^2\lambda_1 A_\alpha^- + \mathcal{E}_\alpha^-.
\]
We then have by direct calculation that
\[
\partial_t \mathcal{N}(t;a) = -\bigl\{2\alpha^2\lambda_1 - \delta\alpha|\xi|\bigr\}\mathcal{N}(t;a) + 2\sqrt{\mathcal{N}(t;a)} e^{\frac12\delta\alpha|\xi|t} \mathcal{E}_\alpha^-(t). 
\]
Now, the arguments in Proposition~\ref{P:bootstrap} showed that the bound \eqref{bound12} on $z$ leads to the pointwise bounds \eqref{bound22-pf} on $|\dot\gamma|$ and $|\dot y|$.  Together with Lemma~\ref{L:WI}, which controls the eigenfunction-potential interaction, these bounds were used to obtain the bound $|\mathcal{E}_\alpha^+(t)|\lesssim e^{-\frac23\delta\alpha|\xi|t}$ in the proof of Proposition~\ref{P:bootstrap}.  The same arguments apply here to give $|\mathcal{E}_\alpha^-(t)| \lesssim e^{-\frac23\delta\alpha|\xi|t}$.  Inserting this estimate into the preceding display then yields \eqref{DIN}. 

Using \eqref{DIN} and the second parameter constraint in \eqref{Pdelta}, we can now observe the following fact: if $T\geq T_0$ is a time at which we have \eqref{bound1}, \eqref{bound12}, and $\mathcal{N}(T;a)=1$, then 
\begin{equation}\label{DP}
\tfrac{d}{dt}\mathcal{N}(t;a)\big|_{t=T} \leq -\tfrac12 \alpha^2\lambda_1<0.
\end{equation}
It follows that if $a\in\partial I^*$, then $\tau(a) = T^*$.  Indeed, the bounds \eqref{bound1} and \eqref{bound12} hold at $t=T^*$ by construction.  Thus as $\mathcal{N}(T^*;a)=1$ for $a\in\partial I^*$, the property \eqref{DP} guarantees that `stability' fails immediately backwards in time.  In particular, the only stable interval for $u$ is $[T^*,T^*]$, which forces $\tau(a)=T^*$. 

We next establish continuity of the map $\tau:I^*\to(T_0,T^*]$. We fix $a_*\in I^*$ and consider the function $f(t,a) = \mathcal{N}(t;a)-1$, which (using well-posedness of \eqref{NLS}) may be defined in a sufficiently small neighborhood of $(\tau(a_*),a_*)$.  Observing that $f(\tau(a_*),a_*)=0$ and applying \eqref{DP} at $(t,a)=(\tau(a_*),a_*)$, the Implicit Function Theorem implies the existence of a continuous function $\sigma$ defined in a neighborhood of $a_*$ such that $f(t,a)=0$ if and only if $t=\sigma(a)$.  Recalling \eqref{Ntaua}, we find that $\tau(a)=\sigma(a)$ for $a$ near $a_*$, so that $\tau$ is continuous at $a_*$. 

It follows that the map $h:I^*\to \R$ defined by $h(a)=A_\alpha^-(\tau(a))$ is continuous and satisfies $h(a)=a$ for $a\in\partial I^*$. In particular, by the Intermediate Value Theorem, there must exist $a_*\in (I^*)^\circ$ such that $h(a_*)=0$. However, this implies that $\mathcal{N}(\tau(a_*);a_*)=0$, which contradicts \eqref{Ntaua} and thus completes the proof. \end{proof}

\section{Proof of the main result}\label{S:proof}

We turn to the proof of our main result, Theorem~\ref{T}.  With Proposition~\ref{P:construction} in place, we can follow the argument of \cite{CMM} quite closely. 

\begin{proof}[Proof of Theorem~\ref{T}] We fix $\delta\in(0,\delta_0)$ and $(\alpha,\xi)\in\mathcal{P}_\delta$ and take a sequence $T_n\nearrow\infty$ with $T_n>T_0$ for all $n$. For each $n$, let $u_n:[T_0,T_n]\times\R^3\to\C$ be the solution constructed in Proposition~\ref{P:construction} with $T^*=T_n$. 

We begin by establishing the following $L^2$-tightness property at the fixed time $T_0$:
\begin{equation}\label{tight}
\lim_{M\to\infty}\limsup_{n\to\infty} \int_{|x|>M} |u_n(T_0,x)|^2\,dx = 0.
\end{equation}

We let $\eta>0$ and, using \eqref{mainbd}, choose $n$ sufficiently large and $T_\eta\in(T_0,T_n)$ sufficiently large that 
\[
\|u_n(T_\eta) - \mathcal{T}_{(T_\eta;\alpha,\xi,0,0)}Q_\alpha\|_{L^2}<\eta. 
\]
We now choose $M_0=M_0(\eta,\alpha,\xi)$ such that
\[
\|\mathcal{T}_{(T_\eta;\alpha,\xi,0,0)}Q_\alpha\|_{L^2(|x|>M_0)} < \eta. 
\]
Combining the previous two inequalities, we obtain
\begin{equation}\label{tight2}
\|u_n(T_\eta)\|_{L^2(|x|>M_0)} \lesssim \eta. 
\end{equation}

We now introduce a smooth radial $\chi$ function that cuts off to the region $|x|>K$, where $K$ will be determined below, and we compute using \eqref{NLS}:
\begin{align*}
\biggl|\tfrac{d}{dt} \int |u_n(t,x)|^2 \chi(x)\,dx \biggr| & = \biggl| 2\Im \int \bar u_n \nabla u_n\cdot \nabla \chi \,dx \biggr| \\
& \lesssim K^{-1} \|u_n\|_{L_t^\infty H_x^1([T_0,T_n]\times\R^3)}^2 \lesssim K^{-1}
\end{align*}
for any $t\in [T_0,T_n]$, where we have used \eqref{mainbd} to control the $L_t^\infty H_x^1$-norm. Thus, applying the Fundamental Theorem of Calculus and choosing $K>M_0$ sufficiently large, we obtain 
\[
\|u_n(T_0)\|_{L^2(|x|>K)}^2 \lesssim \|u_n(T_\eta)\|_{L^2(|x|>K)}^2 + K^{-1}(T_\eta-T_0) \lesssim \eta^2, 
\]
which yields \eqref{tight}. 

Next, using $H^1$-boundedness and passing to a subsequence, we see that there exists $v_0\in H^1$ such that $u_n(T_0)\rightharpoonup v_0$ weakly in $H^1$ as $n\to\infty$.  In fact, using the Rellich--Kondrachov Theorem and \eqref{tight}, we obtain that $u_n(T_0)\to v_0$ strongly in $L^2$. Appealing to $H^1$-boundedness once again, we may upgrade this to strong $H^s$ convergence for any $s\in[0,1)$. 

Now let $v:I_{\max}\times\R^3\to C$ the maximal-lifespan solution to \eqref{NLS} with $v(T_0)=v_0$. Now fix $t\in[T_0,\infty)$ and observe that for all $n$ sufficiently large, the solution $u_n$ exists at time $t$. Thus, by well-posedness for \eqref{NLS} and the fact that $u_n(T_0)\to v_0$ in $H^s$, we have that $v$ must exist at time $t$, with $u_n(t)\to v(t)$ in $H^s$ for $s\in[0,1)$. Moreover, by $H^1$ boundedness, we may further obtain $u_n(t)\rightharpoonup v(t)$ weakly in $H^1$.  Using weak lower-semicontinuity and \eqref{mainbd}, we therefore find
\[
\|v(t) - \mathcal{T}_{(t;\alpha,\xi,0,0)}Q_\alpha\|_{H^1} \leq \liminf_{n\to\infty} \|u_n(t) - \mathcal{T}_{(t,\alpha,\xi,0,0)}Q_\alpha\|_{H^1} \lesssim e^{-\frac12\delta\alpha|\xi| t}. 
\]
This allows us to obtain uniform $H^1$ bounds for $v(t)$ on $[T_0,\infty)$ (so that $v$ is forward-global) and yields the convergence asserted in \eqref{H1C}.\end{proof}


\end{document}